\documentclass[12pt, reqno]{amsart}
\usepackage{hyperref}
\usepackage{amssymb,amsmath,graphicx,amsthm}
\def\simge{\underset\sim>}

\def\T{\text}

\def\1#1{\overline{#1}}

\def\2#1{\widetilde{#1}}

\def\3#1{\widehat{#1}}

\def\4#1{\mathbb{#1}}

\def\5#1{\frak{#1}}

\def\6#1{{\mathcal{#1}}}

\def\C{{\4C}}

\def\R{{\4R}}

\def\La{\Lambda}

\def\T{\text}
\newcommand{\Om}{\Omega}
\newcommand{\om}{\omega}

\newcommand{\no}[1]{\|{#1}\|}

\def\R{{\Bbb R}}

\def\C{{\Bbb C}}

\def\di{\partial}
\def\dib{\bar\partial}
\def\Label#1{\label{#1}}
\def\simge{\underset\sim>}

\def\T{\text}

\def\1#1{\overline{#1}}

\def\2#1{\widetilde{#1}}

\def\3#1{\widehat{#1}}

\def\4#1{\mathbb{#1}}

\def\5#1{\frak{#1}}

\def\6#1{{\mathcal{#1}}}

\def\C{{\4C}}

\def\R{{\4R}}

\def\B{\6B}

\def\P{\6P}

\def\La{\Lambda}

\numberwithin{equation}{section}

\def\T{\text}

\theoremstyle{plain}

\newtheorem{theorem}{Theorem}[section]

\newtheorem{corollary}[theorem]{Corollary}

\newtheorem{lemma}[theorem]{Lemma}

\newtheorem{proposition}[theorem]{Proposition}

\theoremstyle{definition}

\newtheorem{definition}[theorem]{Definition}

\theoremstyle{remark}

\newtheorem{remark}[theorem]{Remark}

\begin{document}

\title[Boundary regularity of the solution...]{Boundary regularity of the solution to the Complex Monge-Amp\`{e}re equation on pseudoconvex domains of infinite type}
  \author[L.~K.~Ha and T.~V.~Khanh]{Ly Kim Ha and Tran Vu Khanh}
\address{Ly Kim Ha}
\address{Dipartimento di Matematica, Universit\`a Degli Studi di Padova, via 
Trieste 63, 35121 Padova, Italy}
\email{ lykimha35@yahoo.com.vn}
\address{Tran Vu Khanh}
\address{Department of Mathematics , National University of Singapore,
Blk S17, 10 Lower Kent Ridge Road, Singapore 119076, Singapore}
\email{mattvk@nus.edu.sg}
\begin{abstract} Let $\Om$ be a bounded, pseudoconvex domain of $\C^n$ satisfying the ``$f$-Property". The $f$-Property is a consequence of the geometric ``type"  of the boundary; it  holds for all pseudoconvex domains of finite type but may also occur for many relevant classes of domains of infinite type. In this paper, we prove the existence, uniqueness and  ``weak" H\"older-regularity up to the boundary of the solution to the Dirichlet problem for the complex Monge-Amp\`{e}re equation
\begin{eqnarray}
\begin{cases}
\det\left[\dfrac{\di^2(u)}{\di z_i\di\bar z_j}\right]=h\ge 0 & \T{in  }\quad\Om, \nonumber\\
u=\phi & \T{on } \quad  b\Om.
\end{cases}
\end{eqnarray}   The idea of our proof goes back to Bedford and Taylor's \cite{BT76}. However, the basic geometrical ingredient is based on a recent result by Khanh \cite{Kha13}.
\end{abstract}

\maketitle
%\tableofcontents 
\section{Introduction}
Let $\Om$ be a bounded, weakly pseudoconvex domain of $\C^n$ with $C^2$-smooth boundary $b\Om$. For given functions $h\ge 0$ defined in $\Om$ and $\phi$ defined on $b\Om$, the Dirichlet problem for the complex Monge-Amp\`{e}re equation consists in finding a continuous, plurisubharmonic function $u$ on $\Om$ such that
\begin{eqnarray}\Label{CMA}
\begin{cases}
\det[u_{ij}]=h& \T{in  }\Om,\\
u=\phi & \T{on } b\Om,
\end{cases}
\end{eqnarray}
where $u_{ij}=\dfrac{\di^2 u}{\di z_i\bar z_j }$ is the $(i,j)^{th}$-entry of  $n\times n$-matrix $[u_{ij}]$. When $u$ is not $C^2(\Om)$,  the first equation in \eqref{CMA} means that $(dd^cu)^n=hdV$ in the sense of Bedford-Taylor \cite{BT76}  (where $dV$ is the Lebesgue measure on $\C^n$). \\

When $\Om$ is a smooth, bounded, strongly pseudoconvex domain in $\C^n$, a great deal of work has been done about the existence, uniqueness and regularity of the solution to the  complex Monge-Amp\`ere problem. The most general related  results  are those  obtained in \cite{BT76} and \cite{CKNS85}.

\begin{itemize}
\item In \cite{BT76}, Bedford and Taylor establish the classical solvability of the Dirichlet problem \eqref{CMA}. Via pluripotential theory \cite{Kl91}, the right hand side is developed in the sense of positive currents when $u$ is continuous, plurisubharmonic. The authors prove that if $\Omega$ is a strongly pseudoconvex, bounded domain in $\mathbb{C}^n$ with $C^2$ boundary, and if $\phi\in Lip^{\alpha}(b\Omega)$, $0\le h^{\frac{1}{n}}\in Lip^{\frac{\alpha}{2}}(\overline{\Omega})$, where $0<\alpha \le 1$, then there is a unique solution $u\in Lip^{\frac{\alpha}{2}}(\overline{\Omega})$ of \eqref{CMA}. This result is sharp. 

\item In \cite{CKNS85}, the smoothness of the solution of \eqref{CMA} is also established. In particular, on a bounded strongly pseudoconvex domain with smooth boundary, if $\phi\in C^{\infty}(b\Omega)$, then there exists a unique solution $u\in C^{\infty}(\overline{\Omega})$ when $h$  is smooth and strictly positive on $\overline{\Om}$. The approach of \cite{CKNS85} follows the continuity method applied to the real Monge-Amp\`{e}re equations \cite{GT83}.
\end{itemize}
 When $\Om$ is not strongly pseudoconvex, there are some known results for the existence and regularity for this problem due to Blocki \cite{Blo96}, Coman \cite{Co97} and Li \cite{Li04}. 
\begin{itemize}
\item In \cite{Blo96}, Blocki also considers the Dirichlet problem \eqref{CMA} on a hyperconvex domain. He proves that when the datum $\phi\in C(b\Omega)$ can be continously extended to a plurisubharmonic function on $\Omega$ and the right hand is nonnegative, continuous, then the plurisubharmonic solution exists uniquely and continuously. However, the H\"older continuity for the solution on these domains is not verified.

\item In \cite{Co97}, Coman shows how to connect some geometrical conditions on a domain in $\mathbb{C}^2$ to the existence of a plurisubharmonic upper envelope in H\"older spaces. In particular, the weak pseudoconvexity of finite type $m$ in $\C^2$ and the fact that the Perron-Bremermann function belongs to $Lip^{\frac{\alpha}{m}}$ with corresponding data in $Lip^{\alpha}$ are equivalent. Again, this means that the finite type condition  plays a critical role in the H\"older regularity of the solution to the  complex Monge-Amp\`{e}re equation.

\item Li \cite{Li04} studies the problem on a domain admitting a non-smooth, uniformly and strictly plurisubharmonic defining function. In particular, if $\Om$ admits a uniformly and strictly plurisubharmonic defining function in $Lip^{\frac{2}{m}}(\bar\Om)$ when $0<\alpha \le \frac{2}{m}$, and $\phi\in Lip^{\alpha}(b\Omega)$ and if $ h^{\frac{1}{n}}\in Lip^{\frac{\alpha}{m}}(\bar\Omega)$, then the  solution  $u\in Lip^{\alpha}(\overline{\Omega})$ of \eqref{CMA} exists uniquely. Based on results by Catlin  \cite{Cat89} and by Fornaess-Sibony  \cite{FS89}, there exists a plurisubharmonic defining function in $Lip^{\frac{2}{m}}(\bar\Om)$ on pseudoconvex domains of finite type $m$ in $\C^2$ or convex domains of finite type $m$ in $\C^n$. 
\end{itemize}
%\begin{remark}
%{\bf In the theory of complex Monge-Amp\`ere equations, an important branch must be mentioned is stutying the Dirichlet problems on compact K\"ahler manifolds with boundary. These such problems can be refered to influential works by Phong et.al. \cite{PS06, PS10, PSS12}.} In these works, for any $0<\alpha<1$, estimates in $C^{\alpha}$ and $C^{1,\alpha}$ of the solution are studied.
%\end{remark}
The main purpose in this paper is to generalize the above results to a pseudoconvex domain, not necessarily of finite type, but admitting an $f$-Property. The $f$-Property consists in the existence of a bounded family of weights in the spirit of \cite{Cat87} and it is sufficient for an $f$-estimate for the $\dib$-Neumann problem \cite{Cat87, KZ10}. We also notice that when $\underset{t\to \infty}{\lim}\dfrac{f(t)}{\log t}=\infty$ the solution of the $\dib$-Neumann problem is regular \cite{Koh02, KZ12b}.
\begin{definition}
\Label{d1}  For a smooth, monotonic,
increasing function $f :[1+\infty)\to[1,+\infty)$ with $\dfrac{f (t)}{t^{1/2}}$ decreasing, we say that $\Om$ has an $f$-Property if there exist a neigborhood $U$ of $b\Om$ and a family of functions $\{\phi_\delta\}$ such that
\begin{enumerate}
  \item [(i)] the functions $\phi_\delta$ are plurisubharmonic,  $C^2$ on $U$, and satisfy $-1\le \phi_\delta \le0$, and
  \item[(ii)] $i\di\dib \phi_\delta\simge f(\delta^{-1})^2Id$ and $|D\phi_\delta|\lesssim  \delta^{-1}$ for any  $z\in U\cap \{z\in \Om:-\delta<r(z)<0\}$, where $r$ is a defining function of $\Om$.
  \end{enumerate}
\end{definition} 
  Here and in what follows, $\lesssim$ and $\simge$ denote inequalities up to a positive constant. Morever, we will use $\approx$ for the combination of $\lesssim$ and $\simge$. \\

\begin{remark}\Label{R1}
 For a pseudoconvex domain,  the $f$-Property is a consequence of the geometric finite type.  In \cite{Cat83, Cat87}, Catlin proves that every smooth, pseudoconvex domain $\Om$ of finite type $m$ in $\C^n$ has the $f$-Property for $f(t)=t^\epsilon$ with $\epsilon=m^{-n^2m^{n^2}}$. In particular, if $\Om$ is  strongly pseudoconvex,  or else it is pseudoconvex of finite type in $\C^2$, or else decoupled or convex in $\C^n$ then $\epsilon=\dfrac{1}{m}$ where $m$ is the  type (cf. \cite{Cat89, Kha10,  McN91b, McN92b}). 
\end{remark}
\begin{remark}\Label{R2}
  The relation of the general type (both finite and infinite type) and the $f$-Property has been studied by Khanh and Zampieri \cite{Kha10, KZ12b}. Moreover, they prove that if $P_1, ...,P_n:\C\to \R^+$ are functions such that $\Delta P_j(z_j)\simge \dfrac{F(|x_j|)}{x_j^2}$ or  $\dfrac{F(|y_j|)}{y_j^2}$ for any $j=1,..., n$, then the pseudoconvex ellipsoid
 $$C=\{(z_1,\dots, z_n)\in\C^n: \sum_{j=1}^nP_j(z_j)\le 1\}$$ 
has the $f$-Property for $f(t)=(F^*(t^{-1}))^{-1}$. Here we denote $F^*$ is the inverse function to $F$.
\end{remark}
In this paper, using the $f$-Property we prove the  ``weak" H\"older regularity for the solution of the Dirichlet problem of complex Monge-Amp\`ere equation. For this purpose we recall the definition of the $f$-H\"older spaces in \cite{Kha12}.\\
\begin{definition}
Let $f$ be an increasing function such that $\underset{t\to+\infty}{\lim} f(t)=+\infty$. For $\Om\subset \C^n$, define the $f$-H\"older space on $\overline{\Om}$ by
$$\Lambda^{f}(\overline{\Om})=\{u : \no{u}_{\infty}+\sup_{z,w\in \overline\Om}f(|z-w|^{-1}) \cdot |u(z)-u(w)|<\infty \}$$ 
and set 
$$\no{u}_{f}= \no{u}_{\infty}+\sup_{z,w\in \overline\Om}f(|z-w|^{-1})\cdot |u(z)-u(w)|. $$
\end{definition}
Note that the notion of the $f$-H\"older space includes the standard H\"older space $\Lambda_\alpha(\overline{\Om})$ by taking $f(t) = t^{\alpha}$ (so that $f(|h|^{-1}) = |h|^{-\alpha}$) with $0<\alpha<1$. The main result in this paper consists in the following:
\begin{theorem}\Label{mainresult} Let $f$ satisfy $g(t)^{-1}:=\displaystyle\int_t^\infty \dfrac{da}{a f(a)}<\infty$. Assume that $\Omega$ is a bounded, pseudoconvex domain  admitting the $f$-Property. Then, for any $0<\alpha\le 1$, if $\phi\in \Lambda^{t^{\alpha}}(b\Omega)$, and $h\ge 0$ on $\Omega$ with $h^{\frac{1}{n}}\in \Lambda^{g^{\alpha}}(\overline{\Omega})$, then the  Dirichlet problem for the complex Monge-Amp\`ere equation
\begin{eqnarray}\label{MA}
\begin{cases}
\det (u_{ij})=h & \T{in  }\quad\Om,\\
u=\phi & \T{on } \quad b\Om,
\end{cases}
\end{eqnarray}
has a unique plurisubharmonic solution $u\in \Lambda^{g^{\alpha}}(\overline{\Omega})$.
\end{theorem}
By Remark~\ref{R1} and \ref{R2}, we immediately have the following
\begin{corollary}
\Label{c1}1) Let $\Om$ be a bounded, $C^2$-boundary, pseudoconvex domain of finite type $m$  in $\C^n$  satisfying at least one of the following conditions: $\Om$ is strongly pseudoconvex, or $\Om$ is convex, or $n=2$, or $\Om$ is decoupled. For any $0<\alpha\le 1$, if $\phi\in Lip^{\alpha}(b\Omega)$, and $h\ge 0$ on $\Omega$ with $h^{\frac{1}{n}}\in Lip^{\frac{\alpha}{m}}(\overline{\Omega})$, then \eqref{MA} has a unique plurisubharmonic solution $u\in Lip^{\frac{\alpha}{m}}(\overline{\Om})$. If $\Omega$ has finite type $m$, but does not satisfy any one of the above additional conditions, at least we have $u\in Lip^{\alpha\epsilon}(\overline{\Om})$ for  $\epsilon=m^{-n^2m^{n^2}}$.\\

2) Let $\Om$ be a complex ellipsoid defined by
$$\Om=\{z=(z_1,...,z_n)\in \C^n:\sum_{j=1}^{n}\exp\left(1-\frac{1}{|z_j|^{s_j}}\right)<1\}.$$
If $s:=\max_{j=1,\dots,n}\{s_j\}<1$, then under the assumption of $\phi$, $h$ and $u$ in Theorem~\ref{mainresult}, we have $u\in  \Lambda^{g^{\alpha}}(\overline{\Omega})$ where  $g(t)=\log^{\frac1s-1}t$.
\end{corollary} 
We organize the paper as follows. In Section 2, we  construct a weak H\"older, uniformly and strictly plurisubharmonic defining function via the work of the second author about the existence of the bumping functions and the plurisubharmonic peak functions on a domain which enjoys  the $f$-Property of \cite{Kha13}. 
This particular defining function is the crucial point in the establishing the existence of the solution to the complex Monge-Amp\`ere equation. Following the work by Bedford-Taylor \cite{BT76}, we prove Theorem~\ref{mainresult} in Section~3. Finally, in Section 4, we introduce an example of infinite type domain which the index function $g$ in Theorem~\ref{mainresult} is sharp.

\section{The $f$-Property }
In this section, under the $f$-Property assumption we construct a uniformly  and strictly plurisubharmonic defining function with $g^2$-H\"older, where $g$ defined in the following theorem.
\begin{theorem}\Label{defining}Let $f$ be in  Definition \ref{d1} such that $g(t)^{-1}:=\displaystyle\int_t^\infty \dfrac{da}{a f(a)}<\infty$. Assume that $\Omega$ is a bounded, pseudoconvex domain  admitting the $f$-Property. Then there exists a strictly plurisubharmonic defining function of $\Om$ which belongs to $g^2$-H\"older space of $\overline{\Om}$, that means, there is a plurisubharmonic function $\rho$ such that 
\begin{enumerate}
  \item $z\in\Om$ if and only if $\rho(z)<0$, $b\Om=\{z\in\C^n: \rho(z)=0\}$ ;
  \item $i\di\dib\rho(X,\bar X)\ge |X|^2$ on $\Om$ in the distribution sense, for any $X\in T^{1,0}\C^n$; and
  \item $\rho$ is in $g^2$-H\"older space of $\overline{\Om}$, that is, $|\rho(z)-\rho(z')|\lesssim g(|z-z'|^{-1})^{-2}$ for any $z, z'\in \overline{\Om}$.
  \end{enumerate} 
\end{theorem}
\begin{remark}\Label{R2a}We note that if $\Om$ is strongly pseudoconvex then $f(t)\approx t^{1/2}$ and hence $g(t)\approx t^{1/2}$. In this case, it is easy to choose a defining function satisfying this theorem. So in the following we only consider that $\Om$ is not strongly pseudoconvex which implies $\dfrac{f(t)}{t^{1/2}}$ is strictly decreasing.   
\end{remark}
The proof of Theorem \ref{defining} is based on the following result about the  existence of  a family of plurisubhamonic peak functions which is recently proven by Khanh \cite{Kha13}. 
 
\begin{theorem}\Label{pshpeak} Under the assumptions of Theorem~\ref{defining}, for any $\zeta\in b\Om$, there exists a $C^2$-plurisuhharmonic function $\psi_\zeta$ on $\Om$ which is continuous on $\overline{\Om}$, and peaks at $\zeta$ (that means, $\psi_\zeta(z)< 0$ for all $z\in\overline{\Om}\setminus\{\zeta\}$ and $\psi_\zeta(\zeta)=0$). Moreover, for any constant $0<\eta<1$, there are  some positive constants $c_1,c_2$ such that the followings hold
 \begin{enumerate}
   \item $|\psi_\zeta(z)-\psi_\zeta(z')|\le c_1|z-z'|^\eta$ for any $z, z'\in\overline{\Om}$; and
      \item $g\big((-\psi_\zeta(z))^{-1/\eta}\big)\le c_2|z-\zeta|^{-1}$ for any $z\in\overline{\Om}\setminus\{\zeta\}$.
 \end{enumerate}
 \end{theorem}

Before giving the proof of Theorem \ref{defining}, we need the following technique lemma
\begin{lemma}\Label{l2}Let $g$ and $\eta$ be in Theorem~\ref{pshpeak}. For $\delta\in (0,1)$, let $\omega(\delta):=g(\delta^{-\frac{1}{\eta}})^{-2}$. Then we have
\begin{enumerate}
  \item[(i)] $\omega$ is increasing function on $(0,1)$ and $\lim_{\delta\to 0^+}\omega(\delta)=0$;
  \item[(ii)] for a suitable choice of $\eta>0$, $\omega$ is concave upward on $(0,1)$;
  \item[(iii)] the inequality $$|\omega(\delta)-\omega(\delta')|\le \omega(|\delta-\delta'|)$$
  holds for any $\delta,\delta'\in (0,1)$; and
\item[(iv)] for a constant $c>0$, there is $c'>0$ such that $\omega(c\delta)\le c'\omega(\delta)$ for $\delta\in(0,1)$.
\end{enumerate}
\end{lemma}
\begin{proof} We first give some calculations on function $g$.  By the definition of $g$, i.e., $\dfrac{1}{g(t)}:=\displaystyle\int^\infty_t\dfrac{da}{af(a)}<\infty$, we have 
\begin{eqnarray}
 \frac{\dot g(t)}{g(t)}=\frac{g(t)}{tf(t)},
\end{eqnarray}
and 
\begin{eqnarray}\Label{2.2}
\frac{\ddot{g}(t)}{\dot g(t)}=\frac{2\dot g(t)}{g(t)}-\frac{1}{t}-\frac{\dot f(t)}{f(t)}.
\end{eqnarray} Since $\dfrac{f(t)}{t^{1/2}}$ is strictly decreasing on $(1,+\infty)$ (by Remark~\ref{R2a}), we obtain $\dfrac{t\dot{f}(t)}{f(t)}<\dfrac{1}{2}$ and 
\begin{eqnarray}\begin{split}\frac{f(t)}{g(t)}=f(t)\int_t^\infty\frac{da}{af(a)}=f(t)\int_t^\infty\frac{a^{1/2} }{f(a)}\cdot\frac{da}{a^{3/2}}\nonumber\ge f(t)\frac{t^{1/2} }{f(t)} \int_t^\infty\frac{da}{a^{3/2}}=2, 
\end{split}
\end{eqnarray}
i.e., $\dfrac{g(t)}{f(t)}\le \dfrac{1}{2}$. Then 
\begin{eqnarray}\Label{new}
\frac{g(t)}{f(t)}+\frac{t\dot f(t)}{f(t)}<1.
\end{eqnarray}

Now we prove this lemma. The proof of (i) immediately follows by the the first derivative of $\om$
$$\dot{\omega}(\delta)=\frac{2}{\eta}\delta^{-\frac{1}{\eta}-1}\dot{g}(\delta^{-\frac{1}{\eta}})g^{-3}(\delta^{-\frac{1}{\eta}})\ge 0.$$
For (ii), we have 
\begin{eqnarray}\Label{ddot}\begin{split}
\ddot{\omega}(\delta)=&-\left(\frac{2}{\eta^2}\delta^{-\frac1\eta-2}\dot{g}(\delta^{-\frac{1}{\eta}})g^{-3}(\delta^{-\frac{1}{\eta}})\right)\left[\eta+1+\frac{\delta^{-\frac{1}{\eta}}\ddot{g}(\delta^{-\frac{1}{\eta}})}{\dot{g}(\delta^{-\frac{1}{\eta}})}-3\frac{\delta^{-\frac{1}{\eta}}\dot{g}(\delta^{-\frac{1}{\eta}})}{g(\delta^{-\frac{1}{\eta}})}\right]\\
=&-\left(\frac{2}{\eta^2}\delta^{-\frac1\eta-2}\dot{g}(\delta^{-\frac{1}{\eta}})g^{-3}(\delta^{-\frac{1}{\eta}})\right)\left[\eta-\frac{\delta^{-\frac{1}{\eta}}g(\delta^{-\frac{1}{\eta}})}{f(\delta^{-\frac{1}{\eta}})}-\frac{\delta^{-\frac{1}{\eta}}\dot{f}(\delta^{-\frac{1}{\eta}})}{f(\delta^{-\frac{1}{\eta}})}\right]\\
\end{split}\end{eqnarray}
where the second equality follows from \eqref{2.2}. From \eqref{new} there is a constant $\eta<1$ such that the bracket term $[\dots]$ in the last line of \eqref{ddot} is non-positive. Therefore, $\omega$ is concave upward.  \\

Now we prove that $|\omega(t)-\omega(s)|\ge \omega(|t-s|) $ for any $t,s\in (0,\delta)$. Assume $t\ge s$, for a fixed $s\in[0,\delta)$ we set $k(t):=\omega(t)-\omega(s)-\omega(t-s)$. Since $\omega$ is concave upward, $\dot{k}(t)=\dot{\omega} (t)-\dot{\omega}(t-s)\ge 0$. That means $k$ is increasing, so we obtain $k(t)\ge k(s)=0$. This completes the proof of (iii). \\
 
For the inequality (iv), we notice that if $c\le1$ then $\omega(ct)\le \omega(t)$ since $\omega$ is increasing. Otherwise, if $c>1$ we use the fact that $\dfrac{g(t) }{t^{1/2}}$ is decreasing (this is obtained from $\dfrac{t\dot g(t)}{g(t)}=\dfrac{g(t)}{f(t)}\le \dfrac{1}{2}$), we have 
$$\omega(ct)=(ct)^{1/2}\frac{\omega(ct)}{(ct)^{1/2}}\le (ct)^{1/2}\frac{\omega(t)}{(t)^{1/2}}=\frac{1}{\sqrt{c}}\omega(t).$$
This completes the proof of Lemma \ref{l2}.
\end{proof}

Now, we will prove the aim of this section.\\
{\em  Proof of Theorem 2.1. } Fix $\zeta\in b\Om$, we define 
$$\rho_\zeta(z)=-\frac{2}{c_2^2}\omega\left(-\psi_\zeta(z)\right)+|z-\zeta|^2,$$
 where $\psi_{\zeta}(.)$ and $c_2$ are in Theorem \ref{pshpeak}.
We will show that the function $\rho_\zeta(z)$ satisfies the following properties:
\begin{enumerate}
  \item $\rho_\zeta(z)<0$, for $z\in\Om$, $\rho_\zeta(\zeta)=0$;
    \item $\rho_\zeta\in C^2(\Om)$ and $i\di\dib\rho_\zeta(X,\bar X)\ge |X|^2$ on $\Om$, and $X\in T^{1,0}\C^n$ ; and
  \item $\rho_\zeta$ is in $g^2$-H\"older space in $\overline{\Om}$.
  \end{enumerate} 
{\em Proof of (1). }From $(2)$ in Theorem \ref{pshpeak} and $\omega$ increasing, we have
\begin{eqnarray}\begin{split}\omega(-\psi_{w}(z))&\ge \omega(G^\eta(|z-w|))\\
&= \left(g\left(G(|z-w|)^{-1}\right)\right)^2\\
&=\left(g(g^*((\gamma|z-w|)^{-1}))\right)^{-2}\\
&=\gamma^2|z-w|^2\ge 0.\end{split}\end{eqnarray}
Hence, as a consequence
$$\rho_{w}(z)=-\frac{2}{\gamma^2}g\bigg((-\psi_{w})^{\frac{-1}{\eta}}(z)\bigg)^{-2}+ |z-w|^2\le -|z-w|^2 <0,$$
where $w\in b\Omega$, and $z\in \Omega$. Moreover, since  $\psi_w(w)=0$ and $\omega(0)=0$, that implies $\rho_w(w)=0$ for any $w\in b\Om$.\\
 
{\em Proof of (2). } Fix $w\in b\Omega$,  the Levi form of $\omega(-\psi_w)$ on $\Om$ is following
\begin{eqnarray}\Label{levig}\begin{split}
i\di\dib \omega(-\psi_w)(X,\bar X)=&\dot{\omega}i\di\dib\psi_w(X,\bar X)-\ddot{\omega}|X\psi_w|^2\ge 0,
\end{split}
\end{eqnarray}
where the inequality follows from Lemma~\ref{l2}(i) and (ii).\\

{\em Proof of (3). } From  Lemma \ref{l2}(iii), we have
\begin{eqnarray}\Label{2.7a}\begin{split}
\left|\omega(-\psi_w(z))-\omega(-\psi_w(z'))\right|\le& \omega\left(|\psi_w(z)-\psi_w(z')|\right)\\
\le&\omega(c|z-z'|^{\eta})\\
\le& c'\omega(|z-z'|^{\eta})=c'g(|z-z'|^{-1})^{-2}.
\end{split}
\end{eqnarray} 
Here the inequalities are obtained from Theorem~ \ref{pshpeak}(1) and Lemma~\ref{l2}(iii)-(iv).\\

On the other hand, since $\Omega$ is bounded and $g(t)\lesssim t^{\frac{1}{2}}$, we can show that
\begin{eqnarray}\Label{2.8a}
||z-w|^2-|z'-w|^2|\lesssim |z-z'|\lesssim g\big(|z-z'|^{-1}\big)^{-2}.
\end{eqnarray}
The inequalities \eqref{2.7a} and \eqref{2.8a} verify that $\rho_w(z)\in \Lambda^{g^2}(\overline{\Omega})$ for uniformly in $w\in b\Omega$.\\

Finally, we define $$\rho(z)=\sup_{w\in b\Om}\rho_w(z).$$
The above properties of $\rho_w$ imply that the function $\rho$ is plurisubharmonic in $\Om$ since the well-known result by LeLong \cite{Le69}, and since $g(0)=0$ and $g : [0,\infty]\to [0,\infty]$, then $\rho$ is also $g^2$-H\"older continuous in $\overline{\Om}$ due to the theory of Modulus of continuity, the superior envelope of these such functions also belongs to the same space. Moreover, since the second property of each $\rho_w$ above, in the distribution sense we have \begin{equation}i\di\dib\rho(L,\bar L)\ge |L|^2.\end{equation}
This completes the proof of Theorem~\ref{defining}. 
$\hfill\Box$

\section{Proof of theorem \ref{mainresult}}

Let $\Omega$ be a bounded open set in $\mathbb{C}^n$, and $\P(\Omega)$ denote the space of plurisubharmonic functions on $\Omega$. The following proof of Theorem \ref{mainresult} is adapted from the argument given by Bedford and Taylor \cite[Theorem 6.2]{BT76} for weakly pseudoconvex domains. Based on the approach in \cite{BT76}, we need the following proposition.

\begin{proposition}\Label{p3.1} Let $\Om$ be a bounded, pseudoconvex domain. Assume that there is a strictly plurisubharmonic defining function $\rho$ of $\Om$ such that $\rho\in \Lambda^{g^2}(\bar\Om)$. Let $0<\alpha \le 1$, and $\phi\in \Lambda^{t^{\alpha}}(b\Omega)$, and let $h\ge 0$ with $h^{1/n}\in \Lambda^{g^{\alpha}}(\overline{\Omega})$. Then, for each $\zeta\in b\Omega$, there exists $v_{\zeta}\in \Lambda^{g^{\alpha}}(\Omega)\cap \P(\Omega)$ such that
\begin{enumerate}
\item[(i)] $v_{\zeta}(z)\le \phi(z)$ for all $z\in b\Omega$, and $v_{\zeta}(\zeta)=\phi(\zeta)$,\\
\item[(ii)] $\no{v_{\zeta}}_{\Lambda^{g^{\alpha}}(\overline{\Omega})}\le C_0$,\\
\item[(iii)] det $(H(v_{\zeta})(z))\ge h(z)$,
\end{enumerate}
where $C_0$ is a positive constant depending only on $\Omega$ and $\no{\phi}_{\Lambda^{t^{\alpha}}(b\Omega)}$.
\end{proposition}

{\it Proof.} For each $\zeta\in b\Omega$, we may choose the family $\{v_\zeta\}$ by
$$v_{\zeta}(z)=\phi(\zeta)-K[-2\rho(z)+|z-\zeta|^2]^{\frac{\alpha}{2}},\quad z\in\overline\Om$$
where $\rho$ is defined by Theorem \ref{defining}, and $K$ will be chosen step by step later.\\

It is easy to see that $v_\zeta(\zeta)=\phi(\zeta)$. Moreover, with the choice $K$ such that
 $K\ge c_{\phi}$, where $c_{\phi}=\displaystyle\sup_{z\ne \zeta\in b\Omega}\dfrac{|\phi(z)-\phi(\zeta)|}{|z-\zeta|^{\alpha}}$, we have
\begin{equation}\label{1}
v_{\zeta}(z)\le -K|z-\zeta|^{\alpha}+ \phi(\zeta)\le \phi(z),\quad \T{for all }~~z\in b\Om.
\end{equation}
This proves (i).\\

For the proof of (ii), we have the following estimates
\begin{eqnarray}\begin{split}
|v_\zeta(z)-v_\zeta(z')| 
&=\bigg|[-2\rho(z)+|z-\zeta|^2]^{\frac{\alpha}{2}}-[-2\rho(z')+|z'-\zeta|^2]^{\frac{\alpha}{2}}\bigg|\\
&\le \bigg|-2\rho(z)+|z-\zeta|^2+2\rho(z')-|z'-\zeta|^2\bigg|^{\frac{\alpha}{2}}\\
&\le \bigg[2|\rho(z)-\rho(z')|+||z-\zeta|^2-|z'-\zeta|^2|\bigg]^{\frac{\alpha}{2}}\\
&\lesssim g^{-\alpha}(|z-z'|^{-1}).\\
\end{split}
\end{eqnarray}
Here, the first inequality follows by the fact that $|t^{\frac{\alpha}{2}}-s^{\frac{\alpha}{2}}|\le |t-s|^{\frac{\alpha}{2}}$ for all $t, s$ small and $\alpha\le1$; the last inequality follows by Theorem~\ref{defining} and \eqref{2.8a}. 
This implies $v_{\zeta}\in \Lambda^{g^{\alpha}}(\overline{\Omega})$ for all $\zeta\in b\Omega$. Moreover $\no{v_\zeta}_{\Lambda^{g^{\alpha}}(\overline{\Omega})}$ is independent on $\zeta$.\\

To assert (iii), we compute $(v_{\zeta})_{ij}$ on $\Om$ 
\begin{eqnarray}\begin{split}
\left(v_{\zeta}(z)\right)_{ij}=&K \frac{\alpha}{2}(-2\rho(z)+|z-\zeta|^2)^{\frac{\alpha}{2}-2}\bigg[(-2\rho(z)+|z-\zeta|^2)(2\rho(z)_{ij}-\delta_{ij})\\
&+\big(1-\frac{\alpha}{2})(-2\rho_i+\bar z_i-\bar \zeta_i)\overline{(-2\rho_j+\bar z_j-\bar \zeta_j)}\big)\bigg].
\end{split}
\end{eqnarray}
Hence 
$$i\di\dib v_{\zeta}(X,X)\ge K \frac{\alpha}{2}(-2\rho(z)+|z-\zeta|^2)^{\frac{\alpha}{2}-1}(2i\di\dib\rho(X,X)-|X|^2)\ge K \frac{\alpha}{2}(-2\rho(z)+|z-\zeta|^2)^{\frac{\alpha}{2}-1}|X|^2,\\
 $$
for any $X\in T^{1,0}\C^n$. Here the last inequality follows from Theorem~\ref{defining}(2). Thus $v_{\zeta}$ is plurisubharmonic and furthermore we obtain
\begin{eqnarray}\begin{split}
\det [(v_{\zeta})_{ij}](z) &\ge \left[K\frac{\alpha}{2}(-2\rho(z)+|z-\zeta|^2)^{\big(\frac{\alpha}{2}-1\big)}\right]^n.\end{split}
\end{eqnarray}
Now, let's choose $$K\ge \max\left\{\frac{2}{\alpha}\max_{z\in\overline{\Om}, \zeta\in b\Om}(-2\rho(z)+|z-\zeta|^2)^{1-\frac{\alpha}{2}}\no{h^{1/n}}_{L^\infty(\Om)},c_{\phi}\right\}.$$
Then
\begin{equation}
\det [(v_{\zeta})_{ij}](z) \ge \no{h^{1/n}}^n_{L^\infty(\Om)}\ge (h^{1/n}(z))^n=h(z),
\end{equation}
for all $z\in \Omega$, and $\zeta\in b\Omega$. This completes the proof of Proposition \ref{p3.1}.

$\hfill\Box$

Before to give a proof of Theorem \ref{mainresult}, we re-call the existence theorem for the problem \eqref{CMA} by Bedford and Taylor \cite[Theorem 8.3, page 42]{BT76}. 
\begin{theorem}[Bedford-Taylor \cite{BT76}]\Label{BT}
Let $\Om$ be a bounded open set in $\C^n$.  Let $\phi\in C(b \Om)$ and $0\le h\in C(\Om)$. If the Perron-Bremerman family 
 \begin{equation*}
\B (\phi, h):=\left\{ v\in \P(\Omega)\cap C(\Omega): \text{det}\, [(v)_{ij}]\ge h\,\text{and}\, \lim\sup_{z\to z_0}v(z)\le \phi(z_0),\,\text{for all $z_0\in b\Omega$} \right\},
\end{equation*}
 is non-empty, and its upper envelope 
 \begin{equation}\Label{sol}
u=\sup\{v: v\in \B(\phi,h)\}
\end{equation}
is continuous on $\bar\Om$ with $u=\phi$ on $ b \Om$, then $u$ is a solution to the Dirichlet problem \eqref{CMA}. 
\end{theorem}

{\it Proof of Theorem \ref{mainresult}.} First, we see that the set $\B(\phi, h)$ is non-empty, in particular, it contains the family of $\{v_{\zeta}\}_{\zeta\in b\Omega}$ in Proposition \ref{p3.1}. The proof of this theorem will be completed if the upper envelope defined in \eqref{sol} has the properties
\begin{enumerate}
  \item $u(\zeta)=\phi(\zeta)$ for all $\zeta\in  b\Om$;
  \item $u\in \La^{g^\alpha}(\bar\Om)$.
  \end{enumerate} 
We note that the uniqueness of solution follows from the Minimum Principle (cf. \cite[Theorem A]{BT76}).\\

Now, we define another upper envelope, for each $z\in \overline{\Omega}$,
$$v(z):=\sup_{\zeta\in b\Omega}\{v_{\zeta}(z)\}.$$
Since the first property of $\{v_{\zeta}\}$ in Proposition \ref{p3.1}, we have 
\begin{eqnarray}\begin{split}
v(\zeta)\ge v_{\zeta}(\zeta)&=\phi(\zeta)\quad\text{for all $\zeta\in b\Omega$},\\
v(z)\le \phi(z)&\quad\text{for all $z\in b\Omega$ ,}
\end{split}
\end{eqnarray}
and so $v=\phi $ on $b\Omega$.\\

 From the second property in Proposition \ref{p3.1}, we have
 $$|v_{\zeta}(z)-v_{\zeta}(z')|\le C_0 (g^{\alpha}(|z-z'|^{-1}))^{-1} \quad\T{for all   } z,z'\in \bar\Om.$$
 Notice that $C_0$ is independent on $\zeta$ so taking the supremum in $\zeta$, theory Modulus of continuity again implies that
$$|v(z)-v(z')|\le C_0 (g^{\alpha}(|z-z'|^{-1}))^{-1} \quad \T{for all   } z,z'\in \bar\Om.$$
By Proposition 2.8 in \cite{BT76}, the following inequality holds
$$\det [(v)_{ij}](z)\ge \inf_{\zeta\in b\Omega} \{\det[(v_{\zeta})_{ij}](z)\}\ge h(z),\quad\text{for all $z\in \Omega$}.$$
Thus, we conclude that $v\in \B(\phi,h)\cap \Lambda^{g^{\alpha}}(\overline{\Omega})$ and $v(\zeta)=\phi(\zeta)$ for any $\zeta\in  b\Om$.\\

By a similar construction there exists a plurisuperharmonic function $w\in \Lambda^{g^\alpha}(\bar\Om)$ such that $w(\zeta)=\phi(\zeta)$ for any $\zeta\in b\Om$. Thus, $v(z)\le u(z)\le w(z)$ for any $z\in \bar\Om$, and hence $u(\zeta)=\phi(\zeta)$ for any $\zeta\in b\Om$. We also obtain
\begin{eqnarray}\Label{new1}
|u(z)-u(\zeta)|\le \max\{\no{v}_{\La^{g^\alpha}(\bar\Om)},\no{v}_{\La^{g^\alpha}(\bar\Om)}\} (g^{\alpha}(|z-\zeta|^{-1})^{-1}\,\T{for any }\, z\in\bar\Om,\zeta\in b\Om.
\end{eqnarray}
Here, the inequality follows by $w,v\in \Lambda^{g^\alpha}(\bar\Om)$ and $v(\zeta)=u(\zeta)=w(\zeta)=\phi(\zeta)$ for any $\zeta\in \di\Om$.\\

Finally, we want to show that \eqref{new1} also holds for $\zeta\in\Om$. For any small vector $\tau\in \mathbb{C}^n$, we define
$$V(z,\tau)=\bigg\{\begin{array}{l}
u(z) \quad\quad\quad\quad\quad\quad\text{if $z+\tau\notin \Omega, z\in \overline{\Omega}$},\\

\max\{u(z), V_{\tau}(z)\},\text{if $z,z+\tau\in \Omega$},
\end{array}$$
where 
$$V_{\tau}(z)=u(z+\tau)+\left(K_1|z|^2-K_2-K_3\right)g^{-\alpha}(|\tau|^{-1})$$
and here
$$K_1\ge\max_{k\in\{1,\dots,n\}}{\binom{n}{k}}^{1/k}\no{h^{\frac{1}{n}}}_{\Lambda^{g^{\alpha}}(\overline{\Omega})},\quad K_2\ge K_1|z|^2,\quad\T{and}\quad K_3\ge \max\{\no{v}_{\Lambda^{g^{\alpha}}(\bar\Omega)},\no{w}_{\Lambda^{g^{\alpha}}(\bar\Omega)}\}.$$
We will show that $V(z,\tau)\in \B(\phi, h)$. Observe that $V(z,\tau)\in \P(\Om)$ for all $z,\tau$. Moreover, for $z\in\di\Om$, $z+\tau\in\Om$, we have 
\begin{eqnarray}\begin{split}
V_\tau(z)-u(z)=&u(z+\tau)-u(z)+\left(K_1|z|^2-K_2-K_3\right)g^{-\alpha}(|\tau|^{-1})\\
\le& \max\{\no{v}_{\La^{g^\alpha}(\bar\Om)},\no{v}_{\La^{g^\alpha}(\bar\Om)}\}g^{-\alpha}(|\tau|^{-1})+\left(K_1|z|^2-K_2-K_3\right)g^{-\alpha}(|\tau|^{-1})\\
\le& 0.
\end{split}\end{eqnarray}
Here the first inequality follows by \eqref{new1} and the second follows by the choices of $K_2$ and $K_3$. This implies that $\lim\sup_{z\to \zeta}V(z,\tau)\le \phi(\zeta)$  for all $\zeta\in b\Om$. For the proof of $\det[V(z,\tau)_{ij}]\ge h(z)$, we need the following lemma:

\begin{lemma}\Label{l3.5} Let $(\alpha_{ij})\ge 0$ and $\beta\in (0,+\infty)$. Then
$$det[\alpha_{ij}+\beta I]\ge \sum^n_{k=0}\beta^k\det(\alpha_{ij})^{(n-k)/n}.$$ 
\end{lemma}
{\it Proof of Lemma~\ref{l3.5}.} Let $0\le\lambda_1\le\cdots\le \lambda_n$ be the eigenvalues of $(\alpha_{ij})$. We have 
\begin{eqnarray}
\begin{split}
\det[\alpha_{ij}+\beta]&=\prod_{j=1}^n(\lambda_j+\beta)\\
&\ge\sum_{k=0}^n\left(\beta^k\prod_{j=k+1}^n\lambda_j\right)\\
&\ge\sum_{k=0}^n\left(\beta^k\det[\alpha_{ij}]^{(n-k)/n}\right).\\
\end{split}
\end{eqnarray}
Here the last inequality follows by $$\det[\alpha_{ij}]=\prod_{j=1}^n\lambda_j\le \left(\prod_{j=k+1}^n\lambda_j\right)^{n/(n-k)}.$$
$\hfill\Box$

Continuing the proof of Theorem \ref{mainresult}, for any $z,z+\tau\in \Om$ we have
\begin{eqnarray}\Label{3.10a}\begin{split}
\det [(V_{\tau}(z))_{ij}]&=\det[{u}_{ij}(z+\tau)+K_1g^{-\alpha}(|\tau|^{-1})I]\\
&\ge \det[{u}_{ij}(z+\tau)]+\sum_{k=1}^n K_1^k [g^{\alpha}(|\tau|^{-1})]^{-k}.\det[{u}_{ij}(z+\tau)]^{\frac{n-k}{n}}\\
&\ge h(z+\tau)+\sum_{k=1}^n K_1^k [g^{\alpha}(|\tau|^{-1})]^{-k}.(h(z+\tau))^{\frac{n-k}{n}},\\
\end{split}
\end{eqnarray}
where the first inequality is derived by Lemma~\ref{l3.5}. Since $h^{\frac1n}\in \Lambda^{g^\alpha}(\Om)$, we obtain
$$h^{\frac1n}(z)-h^{\frac1n}(z+\tau)\le g^{-\alpha}(|\tau|^{-1})\no{h^{\frac1n}}_{\Lambda^{g^\alpha}} ,\quad \T{for any}\quad z,z+\tau\in\Om,$$
and hence 
\begin{eqnarray}\Label{3.11a}
h(z)\le h(z+\tau)+\sum_{k=1}^n\binom{n}{k}h(z+\tau)^{(n-k)/n}\left(g^{-\alpha}(|\tau|^{-1})\no{h^{\frac1n}}_{\Lambda^{g^\alpha}}\right)^k.
\end{eqnarray}
Combining \eqref{3.10a}, \eqref{3.11a} with the choice of $K_1$, we get 
$$\det [(V_{\tau})_{ij}](z)\ge h(z), \quad\T{for any} \quad z,z+\tau\in \Om.$$
We conclude that $V(z,\tau)\in \B(\phi,h)$. It follows that for all $z\in\Om$, $V(z,\tau)\le u(z)$.
If $z,z+\tau\in \Om$, this yields
\begin{eqnarray}\begin{split}
u(z+\tau)-u(z)\le& V(\tau,z)-\left(K_1|z|^2-K_2-K_3\right)g^{-\alpha}(|\tau|^{-1})-u(z)\\
\le&\left(-K_1|z|^2+K_2+K_3\right)g^{-\alpha}(|\tau|^{-1})\\
\le& \left(K_2+K_3\right)g^{-\alpha}(|\tau|^{-1}).
\end{split}
\end{eqnarray}
By reversing the role of $z$ and $z+\tau$, we assert that $u\in\Lambda^{g^{\alpha}}(\overline{\Omega})$. This completes the proof.

$\hfill\Box$
\section{A sharp example}

In this section, we give an example to show that the index function $g$ is sharp. The sharpness of $g$ in the case of strongly pseudoconvex and pseudoconvex of finite type has been considered by Bedford-Taylor \cite{BT76} and Li \cite{Li04}.\\

 Here we consider the problem on the following complex ellipsoid of infinite type in $\mathbb{C}^2$
\begin{equation}\label{Ellip}
E = \left\{  (z_1,z_2)\in \mathbb{C}^2: \rho(z_1,z_2)=\exp(1-\frac{1}{|z_1|^s})+|z_2|^2< 1 \right\},
\end{equation}
where $0<s<1$. It is well-known that this ellipsoid satisfies the $f$-property with $f(t)=(1+\log(t))^{\frac{1}{s}}$ (see \cite{Kha12, KZ10} for details) and hence we define $g(t)=(1+\log(t))^{\frac{1}{s}-1}\approx \left( \displaystyle\int^\infty_t\dfrac{da}{af(a)}\right)^{-1}$. Then for any datum $\phi\in \La^{t^\alpha}(bE)$, $0<\alpha\le 1$, and $h\ge0, h^{1/2}\in \La^{g^\alpha}(\overline{E})$, the complex Monge-Amp\`ere problem \eqref{CMA} has a unique solution in $\La^{g^\alpha}(\bar E)$. We will prove that the index $g^\alpha$ can not be improved, i.e., there does not exist $\tilde g$ such that $\underset{t\to\infty}{\lim}\dfrac{\tilde{g}(t)}{g(t)}=\infty$ and $u\in \La^{\tilde g^\alpha}(\overline{E})$. Indeed, let 
$$u(z)=(1-\log(1-|z_2|^2))^{-\frac{\alpha}{s}}\quad\T{for any ~~}
z\in\overline{E}.$$
It is easy to check that  $u$ is plurisubharmonic and smooth in $E$. Then  $u$ is a solution of the following problem
$$\begin{cases}
\det[u_{ij}]=0&\qquad \T{in}\quad E,\\
u=|z_1|^\alpha&\qquad \T{in}\quad b E,
\end{cases}$$
and hence $u\in \Lambda^{g^{\alpha}}(\bar E)$ by Theorem~\ref{mainresult}. The following claim explains why the index function $g$ can not be improved.\\

{\it Claim: Assume $u\in \Lambda^{\tilde{g}^{\alpha}}(E)$, then $\lim_{t\to\infty}\frac{\widetilde{g}(t)}{g(t)}<\infty$.}\\

{\it Proof of Claim.} For small $\epsilon>0$, let $z_\epsilon=(0,1-\epsilon)$ and $w_\epsilon=(0,1-2\epsilon)$. Since $u\in \Lambda^{\tilde{g}^{\alpha}}(E)$, it follows
\begin{eqnarray}\Label{vd1}
|u(z_\epsilon)-u(w_\epsilon)|\lesssim (\tilde g^{\alpha})^{-1}(|z_\epsilon-w_{\epsilon}|^{-1})=(\tilde g(\epsilon^{-1}))^{-\alpha}.
\end{eqnarray}
By the basis inequality $|x^{-\alpha}-y^{-\alpha}|\ge |x-y|^{-\alpha}$ for any $x,y\in [0,\infty)$ and $\alpha>0$, we obtain 
\begin{eqnarray}\Label{vd2}
|u(z_\epsilon)-u(w_\epsilon)|=|(f(\epsilon^{-2}))^{-\alpha}-(f((2\epsilon)^{-2}))^{-\alpha}|\ge |f(\epsilon^{-2})-f((2\epsilon)^{-2})|^{-\alpha}.
\end{eqnarray}
On the other hand, 
\begin{eqnarray}\Label{vd3}
\begin{split}
f(\epsilon^{-2})-f((2\epsilon)^{-2})=&\int^{\epsilon^{-2}}_{(2\epsilon)^{-2}}f'(t)dt\\
=&\frac{1}{s}\int^{\epsilon^{-2}}_{(2\epsilon)^{-2}}\frac{g(t)}{t}dt\\
\le& \frac{1}{s}\frac{g(t)}{t}\Big|_{t=\epsilon^{-2}}\times \int^{\epsilon^{-2}}_{(2\epsilon)^{-2}}dt\\
\lesssim& g(\epsilon^{-2})\approx g(\epsilon^{-1}),
\end{split}
\end{eqnarray}
where the first inequality follows by $\dfrac{g(t)}{t}=\dfrac{(1+\log t)^{\frac{1}{s}-1}}{t}$ is decreasing in a neighborhood of infinity. From \eqref{vd1}, \eqref{vd2} and \eqref{vd3}, we get 
$$\tilde g(\epsilon^{-1})\lesssim g(\epsilon^{-1})\quad \T{for any} \quad \epsilon>0.$$ This proves the claim and explains why the index function $g$ can not be improved. 

\begin{remark}In \cite{Li04}, the author also  uses this example  to show that for suitable datum $h$ and $\phi$ the unique solution in this example stills be continuous, but it can not belong to any classical H\"older space (i.e., $t^\alpha$-H\"older). 
\end{remark}

\bibliographystyle{alpha}
%\bibliography{Khanh}

\end{document}